\documentclass[12pt, a4paper]{article}

\usepackage{amsmath, amsfonts, amstext, amssymb, amsthm, amsbsy}
\usepackage[T1]{fontenc} 
\usepackage{graphicx}
\usepackage[absolute]{textpos} 
\usepackage{color}
\usepackage{hyperref}
\usepackage[toc,page]{appendix} 
\usepackage{varioref}
\usepackage{enumerate}
\usepackage{multirow}
\usepackage[nottoc, notlof, notlot]{tocbibind}
\usepackage{fancyhdr}
\usepackage{xspace} 

\newtheorem{theorem}{Theorem}
\newtheorem{proposition}{Proposition}

\newtheorem{definition}{Definition}
\newtheorem{lemma}{Lemma}

\newcommand{\cD}{{\mathcal D}}

\newcommand{\cF}{{\mathcal F}}

\newcommand{\bbE}{\mathbb{E}}

\newcommand{\bbP}{\mathbb{P}}
\newcommand{\bbQ}{\mathbb{Q}}
\newcommand{\bbR}{\mathbb{R}}

\newcommand{\bbZ}{\mathbb{Z}}

\begin{document}

\title{Stationary max-stable processes with the Markov property}

\author{Cl\'ement Dombry\footnote{Universit\'e de
Poitiers, Laboratoire de Mathématiques et Applications, UMR CNRS 7348, T\'el\'eport 2, BP 30179, F-86962 Futuroscope-Chasseneuil cedex,
France.  Email: Clement.Dombry@math.univ-poitiers.fr}\ \ and Fr\'ed\'eric Eyi-Minko\footnote{Universit\'e de
Poitiers, Laboratoire de Mathématiques et Applications, UMR CNRS 7348, T\'el\'eport 2, BP 30179, F-86962 Futuroscope-Chasseneuil cedex,
France.  Email: Frederic.Eyi.minko@math.univ-poitiers.fr}}
\date{}
\maketitle

\begin{abstract}
We prove that the class of discrete time stationary max-stable process  satisfying the Markov property is equal, up to time reversal, to the class of stationary max-autoregressive processes of order $1$. A similar statement is also proved for continuous time processes.
\end{abstract}
\vspace{0.5cm}
\noindent
{\bf Key words:} max-stable process, Markov property, max-autoregressive process.
\noindent
{\bf AMS Subject classification. Primary:} 60G70,  {\bf Secondary:} 60J05.

\section{Introduction}
Given a class of stochastic processes, a natural and important question is to determine conditions ensuring the Markov property. 
For example, a zero mean Gaussian process $X$ on $T=\bbZ$ or $\bbR$ satisfies the Markov property if and only if
\[
 \bbE[X(t_2) \mid X(t), t\leq t_1]=\bbE[X(t_2) \mid X(t_1)]\quad \mbox{for all}\ t_1,t_2\in T, t_1\leq t_2.
\]
It is also well known that if $X$ is a {\it stationary} zero mean Gaussian process on $T$ satisfying the Markov property, 
then $X$ must be the Ornstein-Uhlenbeck process with covariance function
\[
 \bbE[X(t_1)X(t_2)]=\bbE[X(0)^2]e^{-\lambda|t_2-t_1|}
\]
for some $\lambda\in [0,+\infty]$. The case $\lambda=0$ corresponds to a constant process, the case $\lambda=+\infty$ to a Gaussian white noise.

Within the class of symmetric $\alpha$-stable (S$\alpha$S) processes, the situation is much more complicated (see Adler {\it et al.} \cite{ACS90}).
No complete characterization of S$\alpha$S Markov processes is known but only necessary or sufficient conditions. One can construct at least two classes
of stationary Markov S$\alpha$S processes, the right and the left S$\alpha$S Ornstein-Uhlenbeck processes.

The purpose of this paper is to study the Markov property within the class of max-stable random processes.
Without loss of generality, we shall consider only {\it simple} max-stable processes defined as follows. 
\begin{definition}
A random process $\eta=(\eta(t))_{t\in T}$ 
is said to be simple max-stable if it has $1$-Fréchet marginals 
\[
\bbP[\eta(t)\leq y]=\exp(-1/y),\quad \mbox{for all}\ t\in T,\ y>0, 
\]
and satisfies the following max-stability property: 
\[
n^{-1} \bigvee_{i=1}^n \eta_i \stackrel{d}= \eta, \quad \mbox{for all}\ n\geq 1, 
\]
where  $(\eta_i)_{i\geq 1}$ are independent copies of $\eta$ and $\bigvee$ denotes pointwise maximum and $\stackrel{d}=$ denotes the equality of distributions.
\end{definition}
Our main result is a complete characterization of the class of stationary simple max-stable Markov processes on $T=\bbZ$ or $\bbR$.
Our analysis relies on a recent paper \cite{DEM13} where explicit formulas for the conditional distributions of max-stable processes are proved. 
This helps clarifying the notion of (Markovian) dependence for max-stable processes.

Related works by Tavares \cite{Tav80}, Alpuim \cite{Alp89}, Alpuim {\it et al.} \cite{AlpAth90} characterized stationary max-AR(1)
processes, and Alpuim {\it et al.} \cite{AlpCatHus95}  study max-autoregressive processes and the Markov property in extreme value theory. Extremes of Markov chains  have been considered by Perfekt \cite{P94} and Smith \cite{Sm92}, while Smith {\it et al.} \cite{STC97} consider Markov chain models for threshold exceedances (see the monograph by Beirlant {\it et al.} section 10.4 for further discussion on extremes and Markov chains).

Well known examples of discrete time simple max-stable processes satisfying  the Markov property are maximum-autoregressive processes of order $1$. The max-AR(1) process with parameter  $a\in [0,1]$
is defined as follows: consider $(F_n)_{n\in\bbZ}$, a sequence of i.i.d. random variables with standard $1$-Fréchet distribution, and set   
\begin{equation}\label{eq:maxAR}
X_a(t)= \left\lbrace 
\begin{tabular}{ll}
  $\bigvee \limits_{n \leq t} (1-a) a^{t-n} F_{n}$ & if $a\in [0,1)$\\
  $X_1(t)\equiv F_0$  &if $a=1$
\end{tabular} \right. 
,\quad t\in \bbZ.
\end{equation}
The max-AR(1) process $X_a$ is a stationary simple max-stable process satisfying 
\[
X_a(t+1)=\max(aX_a(t),(1-a)F_{t+1}),\quad t\in\bbZ.
\]
This relation explains the term max-autoregressive and implies that  $X_a$ satisfies the Markov property.
The associated Markov kernel $K_a(x,\cdot)$ is defined by
\[
K_a(x,dy)=\bbP[X_a(t+1)\in dy \ \mid\ X_a(t)=x]
\]
and is easily computed: denoting by $\delta_z$ the Dirac measure at point $z$, it holds
\[
K_a(x,dy)=e^{-(1-a)/(ax)}\delta_{ax}(dy)+(1-a)y^{-2}e^{-(1-a)/y}1_{\{y>ax\}}dy.
\]
Note that the parameter $a\in[0,1]$ tunes the strength of the dependence, ranging from independence when $a=0$ to complete dependence when $a=1$.
It can be retrieved from the support of the law of $X_a(t+1)/X_a(t)$ since
\[
\mathrm{supp}(X_a(t+1)/X_a(t))=[a, +\infty) \quad \mbox{if}\  a\in [0,1),
\]
and the support is reduced to $\{1\}$ in the case $a=1$.

It is well known that if $X=(X(t))_{t\in\bbZ}$ is a stationary Markov chain, then the time reversed process $\check X=(X(-t))_{t\in\bbZ}$ 
is also a stationary Markov chain. Hence, time reversed max-autoregressive processes are further examples of stationary max-stable Markov processes. 
More precisely, if $X_a$ is the max-AR(1) process \eqref{eq:maxAR}, the associated time reversed process $\check X_a$ is given by
\begin{equation}\label{eq:TRmaxAR}
\check X_a(t)= \bigvee_{n \geq t }  (1-a) a^{n-t} \check F_{n},\quad t\in \bbZ,
\end{equation}
where $(\check F_n)_{n\in\bbZ}=(F_{-n})_{n\in\bbZ}$ are i.i.d. random variables with standard $1$-Fréchet distribution. 
Clearly, $\check X_a$ satisfies the backward max-autoregressive relation
\[
\check X_a(t-1)=\max(a\check X_a(t),(1-a)\check F_{t-1}),\quad t\in\bbZ.
\]
The Markov kernel associated to $\check X_a$ is given by
\begin{eqnarray*}
\check K_a(y,dx)&=&\bbP[\check X_a(t+1)\in dx \ \mid\ \check X_a(t)=y]\\
&=& a\delta_{y/a}(dx)+(1-a)x^{-2}e^{-1/x+a/y}1_{\{x<y/a\}}dx.
\end{eqnarray*}
Note that the Markov kernels $K_a$ and $\check K_a$ are related by the equilibrium relation
\[
\pi(dx) K_a(x,dy)=\pi(dy)\check K_a(y,dx)
\]
where $\pi(dx)=x^{-2}e^{-1/x}1_{\{x>0\}}dx$ is the stationary distribution. 
It is easily seen that for $a=0$ or $a=1$,  $X_a$ and $\check X_a$ have the same distribution. This means that  the max-AR(1) process  $X_a$ is reversible if $a=0$ or $a=1$. 
This is no longer  the case when $a\in (0,1)$ since
\[
 \mathrm{supp}(\check X_a(t+1)/\check X_a(t))=[0, 1/a] \quad \mbox{if}\  a\in (0,1).
\]
The purpose of the present paper 
is the characterization of all  stationary simple max-stable processes satisfying the Markov property.
\begin{theorem}\label{theo1}
Any stationary  simple max-stable process $\eta=(\eta(t))_{t\in\bbZ}$ satisfying the Markov property is equal in distribution to a  max-AR(1) process \eqref{eq:maxAR} or to a time-reversed max-AR(1) process \eqref{eq:TRmaxAR}.
\end{theorem}

The structure of the paper is the following. In section 2, we gather some preliminaries on max-stable processes and their representations that will be useful in our approach. 
Section 3 is devoted to the proof of Theorem \ref{theo1}. An extension to continuous time processes is considered in section 4.

\section{Preliminaries on max-stable processes}

\subsection{Representations of max-stable processes}
Our approach relies on the following representation of simple max-stable process due to de Haan \cite{dH84}, see also Penrose \cite{Pen92} and Schlather \cite{Sc02}.
The symbol $\stackrel{d}=$ stands for equality in distribution.
\begin{theorem}
Let $\eta=(\eta(t))_{t\in\bbZ}$ be a simple max-stable process on $\bbZ$. Then, there exists a nonnegative random process $Y$ such that 
\begin{equation}\label{eq:Y}
\bbE[Y(t)]=1\quad \mbox{for all}\ t\in \bbZ,
\end{equation}
and
\begin{equation}\label{eq:deHaan}
\Big(\eta(t)\Big)_{t\in\bbZ}\stackrel{d}=\Big(\bigvee_{i\geq 1}U_iY_i(t)\Big)_{t\in\bbZ},
\end{equation}
where $(Y_i)_{i\geq 1}$ are i.i.d. copies of $Y$ and $\{U_i,i\geq 1\}$ is a Poisson point process on $(0,+\infty)$ with intensity $u^{-2}du$ and independent of $(Y_i)_{i\geq 1}$.
\end{theorem}
The random process $Y$ is called a spectral process associated to $\eta$. Conversely, we call $\eta$ the max-stable process associated to $Y$.

Consider the function space $\cF=[0,+\infty)^{\bbZ}$ endowed with the product sigma-algebra and $\cF_0=\cF\setminus\{0\}$. The exponent measure of $\eta$ is defined by  
\begin{equation}\label{eq:defmu}
\mu(A)=\int_0^\infty \bbP[uY\in A]u^{-2}dy,\quad A\in\cF_0\ \mbox{measurable}.
\end{equation}
It does not depend on the choice of the spectral process $Y$ but only on the distribution of $\eta$. It satisfies the 
homogeneity property
\[
\mu(uA)=u^{-1}\mu(A),\quad u>0, \ A\in\cF_0\ \mbox{measurable},
\]
and is related to $\eta$ by the relations
\begin{eqnarray*}
&&\bbP[\eta(t_1)\leq z_1,\ldots,\eta(t_k)\leq z_k ]\\
&=&\exp\Big(-\mu\{f\in\cF;\ f(t_i)>z_i\ \mbox{for some}\ 1\leq i\leq k \}\Big)
\end{eqnarray*}
for all $k\geq 1$, $t_1,\ldots,t_k\in\bbZ$, $z_1,\ldots,z_k\geq 0$.

Note that there is no uniqueness for the representation \eqref{eq:deHaan}. We introduce therefore the following notion of equivalent spectral processes.
\begin{definition}\label{def:equiv}
Let $Y$ and $Y'$ be nonnegative stochastic processes satisfying 
\begin{equation}\label{eq:defequiv}
\bbE[Y(t)]=\bbE[Y'(t)]=1,\quad t\in\bbZ.
\end{equation}
We say that $Y$ and $Y'$ are equivalent if and only if  the associated max-stable processes have the same distribution.
\end{definition}

The following property will be useful. A subset $C\subset \cF$ is called a cone if and only if $f\in C$ implies $uf\in C$ for all $u\geq 0$.
\begin{proposition}\label{prop:equiv}
Let $Y$ and  $Y'$ be equivalent processes as in Definition \ref{def:equiv}. Let  $C\subset\cF$ be a measurable cone such that $\bbP[Y\in C]=1$. Then, $\bbP[Y'\in C]=1$.
\end{proposition}

\begin{proof} 
Let $\mu$ (resp. $\mu'$) be the exponent measure of the max-stable process associated to $Y$ (resp. $Y'$) by Equation \eqref{eq:defmu}. 
Clearly,  $Y$ and $Y'$ are equivalent if and only the exponent measures $\mu$ and $\mu'$ are equal.
On the other hand,  Equation \eqref{eq:defmu} implies clearly that $\bbP[Y\in C]=1$ if and only if  $\mu$ is supported by $C$, i.e. $\mu[\cF\setminus C]=0$. Similarly $\bbP[Y'\in C]=1$ if and only if $\mu'[\cF\setminus C]=0$. 

Using this, we deduce easily that if $Y$ and $Y'$ are equivalent processes with $\bbP[Y\in C]=1$, then $\mu=\mu'$ is supported by $C$ and $\bbP[Y'\in C]=1$.
%
\end{proof}

\subsection{Brown-Resnick stationary processes}
In the following we focus on stationary max-stable processes. A random process $X=(X(t))_{t\in\bbZ}$ is called stationary 
if $X$ and $X(\cdot+s)$ have the same distribution for all $s\in\bbZ$. We use the following terminology, due to Kabluchko {\it et al.} \cite{KSdH09}.
\begin{definition}\label{def:BR}
A nonnegative random process $Y$ satisfying \eqref{eq:Y} is called Brown-Resnick stationary if the associated max-stable process $\eta$ defined by \eqref{eq:deHaan} is stationary.
\end{definition}
It follows from the definition that $Y$ is Brown-Resnick stationary if and only if $Y$ and $Y(\cdot+s)$ are equivalent 
(in the sense of Definition \ref{def:equiv}) for all $s\in\bbZ$.  Proposition \ref{prop:equiv} implies then the following result. 

\begin{proposition}\label{prop:BRcone}
Let $Y$ be a Brown-Resnick stationary process and let   $C\subset\cF$ be a measurable cone such that $\bbP[Y\in C]=1$.
Then, for all $s\in\bbZ$, $\bbP[Y(\cdot+s)\in C]=1$. Furthermore, noting $\theta_s:\cF\to\cF$ the shift operator defined by $\theta_s(f)=f(\cdot+s)$, it holds
\[
\bbP\Big[Y\in\bigcap_{s\in\bbZ}\theta_s(C) \Big]=1.
\]
\end{proposition}
\begin{proof}
As we noticed, if $Y$ is Brown-Resnick stationary, then  $Y$ and $Y(\cdot+s)$ are equivalent for all $s\in\bbZ$. 
The result follows then directly from Proposition~\ref{prop:equiv} by setting $Y'=Y(\cdot+s)$. For the last statement, if $\bbP[Y\in C]=1$, then 
$\bbP[Y\in \theta_{s}(C)]=1$ for all $s\in\bbZ$, whence we deduce $\bbP\big[Y\in\cap_{s\in\bbZ}\theta_s(C) \big]=1$.
\end{proof}

The following lemma will also be useful in order to prove equivalence of processes. For $f_0\in\cF$, we note $C_{inv}(f_0)=\{uf_0(\cdot+s);\ u\geq 0, s\in\bbZ\}$ the smallest shift-invariant cone containing $f_0$. 
\begin{lemma}\label{prop:crit}
Let $Y$ and $Y'$ be  Brown-Resnick stationary processes satisfying \eqref{eq:defequiv} and such that
\[
\bbP[Y\in C_{inv}(f_0)]=\bbP[Y'\in C_{inv}(f_0)]=1\quad \mbox{for some}\ f_0\in\cF.
\]
Then $Y$ and $Y'$ are equivalent.
\end{lemma}
This lemma can be related to the notion of stationary indecomposable max-stable process (see Wang {\it et al.} \cite{WSR12}). The fact that the cone $C_{inv}(f_0)$ 
cannot be written as a disjoint union of shift-invariant smaller cones implies that the associated stationary max-stable process is indecomposable.
\begin{proof}
We denote by $\mu$ and $\mu'$ the exponent measure of the max-stable processes associated to $Y$ and $Y'$ respectively.
Clearly the measure $\mu$ (and also $\mu'$) satisfies the following four properties:
\begin{itemize}
\item[i)] $\mu$ is $-1$-homogeneous;
\item[ii)] $\mu(\{f\in\cF;\ f(0)\geq 1\})=1$;
\item[iii)] $\mu$ is shift-invariant;
\item[iv)] $\mu$ is supported by $C_{inv}(f_0)$.
\end{itemize}
Recall indeed that properties i) and ii) are satisfied for all exponent measures, that iii) holds if and only if the spectral process $Y$ is Brown-Resnick stationary and  that iv) is equivalent to  $\bbP[Y\in C_{inv}(f_0)]=1$. 

We will prove below that there exists at most one measure $\mu$ on $\cF_0$ satisfying the four properties i)-iv). Since $\mu'$ satisfies the same properties, we deduce that $\mu=\mu'$, whence $Y$ and $Y'$ are equivalent.

\vspace{0.5cm}

For $g\in\cF$, we denote by $C(g)=\{ug;\ u\geq 0\}$ the smallest cone containing $g$. Clearly,
$C_{inv}(f_0)=\cup_{s\in\bbZ} C(f_s)$ with $f_s=\theta_s f_0$. The different cones in the union may have non trivial intersections and two cases occur.
\begin{itemize}
\item First case: there is $s_0\geq 1$ such that $C(f_{s_0})\cap C (f_0)\neq \{0\}$. \\
Without loss of generality, we can suppose that $s_0$ is minimal with this property. Then, 
\[
C_{inv}(f)=\cup_{0\leq s\leq s_0-1} C(f_s) 
\]
and
\[
C(f_s)\cap C(f_{s'})=\{0\},\quad 0\leq s\neq s'\leq s_0-1.
\]

\item Second case: for all $s\geq 1$, $C(f_s)\cap C (f)= \{0\}$.\\
Then 
\[
C_{inv}(f)=\cup_{ s\in\bbZ} C(f_s) \quad \mbox{and}\quad C(f_s)\cap C(f_{s'})=\{0\},\quad  s\neq s'.
\]
\end{itemize}
We give the proof in the first case only, the second case follows from straightforward modifications. The support property iv) implies that $\mu=\sum_{s=0}^{s_0-1}\mu_s$ where $\mu_s$ is the restriction of $\mu$ to $C(f_s)$. The homogeneity property i) entails that the restriction $\mu_s$ is completely determined by the real parameter $\alpha_s=\mu(\{uf_s; u\geq 1\})$. It holds indeed
\[
\mu_s(\{uf_s; u\geq v\})=v^{-1}\mu_s(\{uf_s; u\geq 1\})=\alpha_sv^{-1},\quad v>0.
\]
Furthermore, the shift invariance property iii) implies that $\alpha_s\equiv \alpha$ does not depend on $s$. Finally, the real parameter $\alpha$ is determined by the normalization property ii): we have indeed
\[
\mu(\{f\in\cF; f(0)\geq 1\})=\sum_{s=0}^{s_0-1}\mu_s(\{f\in\cF; f(0)\geq 1\})=\alpha\sum_{s=0}^{s_0-1}f_0(s) 
\] 
whence property ii) yields $\alpha=(\sum_{s=0}^{s_0-1}f_0(s))^{-1}$. This proves that $\mu$ is completely determined by properties i)-iv) and completes the proof of the lemma. 
\end{proof}

\subsection{Conditional distributions}
Our study of the Markov property for max-stable process relies on explicit formulas for the regular conditional distributions of max-stable process
established in Dombry and Eyi-Minko \cite{DEM13}. The following expression for the conditional distribution function will be useful (see Proposition~4.1 in \cite{DEM13}).
\begin{proposition}\label{prop:distcond}
Let $\eta$ be a simple max-stable process with representation \eqref{eq:deHaan}. For every $t,t_1,\ldots,t_k\in \bbZ$ and $z,z_1,\ldots,z_k>0$ 
\begin{eqnarray*}
&&\bbP[\eta(t_1) \leq z_1,\ldots,\eta(t_k) \leq z_k \mid \eta(t)=z]\\
&=&\bbE\big[1_{\{\vee_{i=1}^k \frac{Y(t_i)}{z_i} \leq  \frac{Y(t)}{z}\}}Y(t) \big] 
 \exp\Big(-\bbE\Big[\Big(\bigvee_{i=1}^k \frac{Y(t_i)}{z_i}-\frac{Y(t)}{z} \Big)^+ \Big]\Big),
\end{eqnarray*}
with $(x)^+=\max(x,0)$.
\end{proposition}

The following well known result on independence for max-stable process will also be useful ({\it cf.} de Haan \cite{dH84}).
\begin{proposition}\label{prop:indep}
Let $\eta$ be a simple max-stable process with representation \eqref{eq:deHaan} and consider $t_1,t_2\in\bbZ$.
Then $\eta(t_1)$ and $\eta(t_2)$ are independent if and only if  
\[
\bbP[Y(t_1)=0 \ \mbox{or}\ Y(t_2)=0]=1.
\]
\end{proposition}

\section{Proof of Theorem~\ref{theo1}}

\subsection{A property of max-stable Markov processes}

The following result is the central tool in our proof of Theorem~\ref{theo1}. Note that no stationarity assumption is required at this stage.

\begin{proposition}\label{theopre}
Let $\eta=(\eta(t))_{t\in\bbZ}$ be a simple max-stable process with representation \eqref{eq:deHaan}. 
For $t,t'\in\bbZ$, we denote by $\alpha_{t,t'}$  the essential infimum of the random variable $Y(t')/Y(t)$ conditionally on $Y(t)>0$, i.e.
\begin{equation}\label{eq:defalpha}
\alpha_{t,t'}=\inf\{c>0;\ \bbP[Y(t')/Y(t)\leq c \mid Y(t)>0]>0\}.
\end{equation}
If $\eta$ satisfies the Markov property, then, for all  $t_1<t<t_2$,
\[
\bbP[Y(t_1)=\alpha_{t,t_1}Y(t) \mid Y(t)>0]=1 \quad \mbox{or} \quad  \bbP[Y(t_2)=\alpha_{t,t_2}Y(t)\mid Y(t)>0]=1. 
\] 
\end{proposition}

\begin{proof}[Proof of Proposition~\ref{theopre}]
First note that the definition \eqref{eq:defalpha} entails 
\[
\bbP[Y(t')\geq \alpha_{t,t'}Y(t)\mid Y(t)>0]=1.
\]
Hence, in order to prove Proposition~\ref{theopre}, it  is enough to prove
\[
\bbP[Y(t_1)\leq \alpha_{t,t_1}Y(t)\mid Y(t)>0]=1 \quad \mbox{or} \quad \bbP[Y(t_2)\leq \alpha_{t,t_2}Y(t)\mid Y(t)>0]=1, 
\]
or equivalently that, for all $c_1>\alpha_{t,t_1}$ and all $c_2>\alpha_{t,t_2}$,
\begin{equation}\label{eq:theopre1}
\bbP[Y(t_1)\leq c_1 Y(t) \mid Y(t)>0]=1 \ \ \mbox{or} \  \ \bbP[Y(t_2)\leq c_2 Y(t) \mid Y(t)>0]=1. 
\end{equation}

We now prove Equation \eqref{eq:theopre1}. We use the fact that the past and the future of a Markov chain are independent conditionally on the present.
More formally, for all  $t_1<t<t_2$ and  $z,z_1,z_2>0$,
\begin{eqnarray*}
&&\bbP[\eta (t_1)\leq z_1, \eta(t_2)\leq z_2 \mid \eta(t)=z]\\
&=&\bbP[\eta (t_1)\leq z_1 \mid \eta(t)=z]\,\bbP[\eta(t_2)\leq z_2 \mid \eta(t)=z].
\end{eqnarray*}
Using the explicit expression for the conditional cumulative distribution function given in Proposition \ref{prop:distcond}, this is equivalent to
\begin{eqnarray*}
&&\bbE\big[1_{\{\frac{Y(t_1)}{z_1} \vee \frac{Y(t_2)}{z_2} \leq  \frac{Y(t)}{z}\}}Y(t)\big]  
 \exp\Big(-\bbE\Big[ \Big(  \frac{Y(t_1)}{z_1} \vee \frac{Y(t_2)}{z_2}-\frac{Y(t)}{z} \Big)^+ \Big]\Big)\\
&=& \bbE\big[1_{\{\frac{Y(t_1)}{z_1} \leq  \frac{Y(t)}{z}\}}Y(t)\big]\exp\Big(-\bbE\Big[ \Big(  \frac{Y(t_1)}{z_1} -\frac{Y(t)}{z} \Big)^+ \Big]\Big) \\
&&\quad \times \ 
\bbE\big[1_{\{\frac{Y(t_2)}{z_2} \leq  \frac{Y(t)}{z}\}}Y(t)\big]\exp\Big(-\bbE\Big[ \Big(  \frac{Y(t_2)}{z_2}-\frac{Y(t)}{z} \Big)^+ \Big]\Big).
\end{eqnarray*}
Using the identity $(a\vee b-c)^+-(a-c)^+-(b-c)^+=(a\wedge b-c)^+$, this last equation simplifies into
\begin{eqnarray*}
&&\bbE\big[1_{\{\frac{Y(t_1)}{z_1} \vee \frac{Y(t_2)}{z_2} \leq  \frac{Y(t)}{z}\}}Y(t)\big]  
 \exp\Big(\bbE\Big[ \Big(  \frac{Y(t_1)}{z_1} \wedge\frac{Y(t_2)}{z_2}-\frac{Y(t)}{z} \Big)^+ \Big]\Big)\\
&=& \bbE\big[1_{\{\frac{Y(t_1)}{z_1} \leq  \frac{Y(t)}{z}\}}Y(t)\big]\bbE\big[1_{\{\frac{Y(t_2)}{z_2} \leq  \frac{Y(t)}{z}\}}Y(t)\big].
\end{eqnarray*}
Finally, setting $z_1=c_1z$ and $z_2=c_2z$ with $c_1,c_2,z>0$, we obtain
\begin{eqnarray*}
&&\bbE\big[1_{\{\frac{Y(t_1)}{c_1} \vee \frac{Y(t_2)}{c_2} \leq  Y(t)\}}Y(t)\big]  
 \exp\Big(\bbE\Big[ \frac{1}{z}\Big(  \frac{Y(t_1)}{c_1} \wedge\frac{Y(t_2)}{c_2}-Y(t) \Big)^+ \Big]\Big)\\
&=& \bbE\big[1_{\{\frac{Y(t_1)}{c_1} \leq  Y(t)\}}Y(t)\big]\bbE\big[1_{\{\frac{Y(t_2)}{c_2} \leq  Y(t)\}}Y(t)\big].
\end{eqnarray*}
Note that the right hand side of this equality does not depend on $z>0$ and is positive as soon as  $c_1>\alpha_{t,t_1}$ and  $c_2>\alpha_{t,t_2}$ 
(this is a simple consequence of the definition \eqref{eq:defalpha}). Then, the exponential factor in the left hand side must be constant and  equal to $1$.
We deduce that, for all $c_1>\alpha_{t,t_1}, c_2>\alpha_{t,t_2}$,
\begin{equation}\label{eq:theopre2}
 \bbE\big[1_{\{\frac{Y(t_1)}{c_1} \vee \frac{Y(t_2)}{c_2} \leq  Y(t)\}}Y(t)\big]= 
\bbE\big[1_{\{\frac{Y(t_1)}{c_1} \leq  Y(t)\}}Y(t)\big]\bbE\big[1_{\{\frac{Y(t_2)}{c_2} \leq  Y(t)\}}Y(t)\big]
\end{equation}
and also
\begin{eqnarray}
\bbP\Big[\frac{Y(t_1)}{c_1} \wedge\frac{Y(t_2)}{c_2}\leq Y(t) \Big]=1\label{eq:theopre3}.
\end{eqnarray}
Let us introduce the  probability   measure  $\hat\bbP_{t}(\cdot)=\bbE[1_{\{\cdot\}}Y(t)]$ and the events 
\[
A_1=\Big\{\frac{Y(t_1)}{c_1} \leq  Y(t) \Big\}\quad \mbox{and}\quad A_2=\Big\{\frac{Y(t_2)}{c_2} \leq  Y(t) \Big\}.
\]
With these notations, Equation \eqref{eq:theopre2} becomes 
\[
\hat\bbP_t[A_1\cap A_2]=\hat\bbP_t[A_1]\hat\bbP_t[A_2] 
\]
and states that the events $A_1$ and $A_2$ are $\hat\bbP_t$-independent. On the other hand, Equation \eqref{eq:theopre3} yields
$\bbP[A_1\cup A_2]=1$ which clearly implies $\hat\bbP_t[A_1\cup A_2]=1$. 
Taking the complementary set, we obtain $\hat\bbP_t[A_1^c\cap A_2^c]=0$ and, from the independence of $A_1$ and $A_2$, 
$\hat\bbP_t[A_1^c]\hat\bbP_t[ A_2^c]=0$. Thus, we  have $\hat\bbP_t[A_1^c]=0$ or $\hat\bbP_t[A_1^c]=0$. 
Finally, the probability measures $\hat\bbP_t[\,\cdot\,]$ and $\bbP[\,\cdot\,\mid Y(t)>0]$ 
are equivalent in the sense that they have the same null sets. Hence, it holds 
\[
\bbP[A_1^c\mid Y(t)>0]=0\quad\mbox{or}\quad \bbP[A_2^c\mid Y(t)>0]=0.
\]
This is equivalent to Equation \eqref{eq:theopre1} and this concludes the proof of Proposition~\ref{theopre}.
\end{proof}

\subsection{A characterization of max-AR(1) processes}
We provide a simple characterization of max-autoregressive processes that will be useful for the proof of Theorem~\ref{theo1}. 
We consider the cone of constant functions  
\[
D_1=\{f\in\cF;\ \forall t\in\bbZ,\ f(t)=f(0)\},
\]
the cone of Dirac functions 
\[
D_0=\{f\in\cF;\ \exists t_0\in\bbZ,\ \forall t\in\bbZ,\ f(t)=f(t_0)1_{\{t=t_0\}}\},
\]
and also, for $a\in (0,1)$, the cone 
\[
D_a=\{f\in\cF;\ \exists t_0\in\bbZ,\ \forall t\in\bbZ,\ f(t)=f(t_0)a^{t-t_0}1_{\{t\geq t_0\}}\}.
\]

\begin{proposition}\label{prop:maxAR}
Let $\eta$ be a simple max-stable process with representation \eqref{eq:deHaan} and assume that $\eta$ is stationary.  Then, the following statements are equivalent:
\begin{itemize}
\item[i)] $\eta$ has the same distribution as the max-AR(1) process $X_a$ defined by \eqref{eq:maxAR},
\item[ii)] $\bbP[Y\in D_a]=1$.
\end{itemize} 
\end{proposition}
 
\begin{proof}
We denote by $\mu'$  the exponent measure of $X_a$. 
For $a\in [0,1)$, Equation \eqref{eq:maxAR} implies that $X_a=\vee_{n\in\bbZ} F_nf_a(\cdot-n)$ with $f_a(t)=(1-a)a^t1_{\{t\geq 0\}}$,
whence we deduce that  $\mu'$ is given by
\[
\mu'[A]=\sum_{n\in\bbZ}\int_0^\infty 1_{\{uf_a(\cdot-n)\in A\}} u^{-2}du,\quad A\subset \cF_0\ \mbox{measurable}.
\]
For $a=1$, it holds $X_a=F_0f_1$ with $f_1(t)\equiv 1$, so that 
\[
\mu'[A]=\int_0^\infty 1_{\{uf_1\in A\}} u^{-2}du,\quad A\subset \cF_0\ \mbox{measurable}.
\]
In both cases, $\mu'$ is clearly supported by the cone of functions $D_a$. This implies that if $Y'$ is a spectral process associated to $X_a$, then $\bbP[Y'\in D_a]=1$.

We now prove the implication $i)\Rightarrow ii)$. If $\eta$ has the same distribution as the max-AR(1) process $X_a$, then the spectral processes $Y$ and $Y'$ are equivalent 
and Proposition~\ref{prop:BRcone} implies $\bbP[Y\in D_a]=1$.

We finally prove the converse implication $ii)\Rightarrow i)$. We assume that $\bbP[Y\in D_a]=1$ and we apply Lemma~\ref{prop:crit}.
Note that $D_a$ is equal to the smallest shift invariant cone containing $f_a$ and denoted by $C_{inv}(f_a)$. The spectral processes $Y$ and $Y'$ are Brown-Resnick stationary processes such that $\bbP[Y\in C_{inv}(f_a)]=\bbP[Y'\in C_{inv}(f_a)]=1$. Lemma~\ref{prop:crit} entails that $Y$ and $Y'$ are equivalent, which means that $\eta$ and $X_a$ have the same distribution.
\end{proof}

\subsection{Proof of Theorem~\ref{theo1}}
Let $\eta$ be a simple max-stable process with representation \eqref{eq:deHaan}. We assume that $\eta$ is stationary and  satisfies the Markov property. 
According to Proposition~\ref{theopre} with $t=0$, $t_1=-1$ and $t_2=1$, it holds
\[
\bbP[Y(-1)=\alpha_{0,-1}Y(0)\mid Y(0)>0]=1\ \mbox{or}\ \bbP[Y(1)=\alpha_{0,1}Y(0)\mid Y(0)>0]=1.
\]
Two cases naturally appear:
\begin{itemize}
\item Case 1:  $\bbP[Y(1)=\alpha_{0,1}Y(0)\mid Y(0)>0]=1$.\\
We will prove below that, in this case, $\eta$ is a max-AR(1) process \eqref{eq:maxAR} with parameter $a=\alpha_{0,1}$. 
To this aim, we use the characterization of max-AR(1) processes given by Proposition~\ref{prop:maxAR} so that it is enough to prove $\bbP[Y\in D_a]=1$.
Note that $a\in[0,1]$ since $a=\alpha_{0,1}=\bbE[Y(1)1_{\{Y(0)>0\}}]$.
\item Case 2: $\bbP[Y(-1)=\alpha_{0,-1}Y(0)\mid Y(0)>0]=1$.\\
We prove that, in this case, $\eta$ is a time reversed max-AR(1) process \eqref{eq:TRmaxAR} with parameter $a=\alpha_{0,-1}$. 
This is easily deduced from case 1 since the time reversed process 
$\check\eta=(\eta(-t))_{t\in\bbZ}$ is a stationary simple max-stable process satisfying the Markov property. 
The associated spectral process $\check Y=(Y(-t))_{t\in\bbZ}$ satisfies $\bbP[\check Y(1)=\alpha_{0,-1}\check Y(0)\mid \check Y(0)>0]=1$, so that $\check\eta$ is a max-AR(1) process with parameter $a=\alpha_{0,-1}$.
\end{itemize}
Thanks to the discussion above, the proof of Theorem~\ref{theo1} is reduced to the proof of the following statement: 
\begin{equation}\label{eq:theo1.1}
\mbox{If}\ \bbP[Y(1)=aY(0)\mid Y(0)>0]=1, \ \mbox{then}\ \bbP[Y\in D_a]=1.
\end{equation} 
We consider three different cases: $a\in (0,1)$, $a=0$ and $a=1$.

\vspace{0.5cm}
\noindent
{\bf Proof of \eqref{eq:theo1.1} in the case $a\in (0,1)$:}\\
We define the cone $C\subset \cF$ by
\[
C=\{f\in\cF;\ f(0)>0\Rightarrow f(1)=af(0)\}.
\]
The property $\bbP[Y(1)=aY(0)\mid Y(0)>0]=1$ implies  $\bbP[Y\in C]=1$. Since $\eta$ is stationary, $Y$ is Brown-Resnick stationary and Proposition \ref{prop:BRcone} implies
\begin{equation}\label{eq:theo1.2}
\bbP\Big[ Y\in\bigcap_{s\in \bbZ} \theta_s(C) \Big]=1.
\end{equation}
Clearly, $\bigcap_{s\in \bbZ} \theta_s(C)$ is equal to the set of functions
\[
\big\{f\in\cF;\ \forall s\in\bbZ,  f(s)>0 \Rightarrow f(s+1)=af(s)\big\}.
\]
For such a function $f$, we easily prove  by induction that  $f(t_0)>0$ implies $f(t)=f(t_0)a^{t-t_0}$ for all $t>t_0$. 
Then, if $t_0=\min\{t\in\bbZ; f(t)>0\}>-\infty$,  $f(t)=f(t_0)a^{t-t_0}1_{\{t\geq t_0\}}$ for all $t\in\bbZ$, and $f\in D_a$. Otherwise, if $t_0=\min\{t\in\bbZ; f(t)>0\}=-\infty$, 
$f(t)=f(0)a^t$ for all $t\in\bbZ$ and $f$ belongs to the cone $D_a'$ generated by the power function $t\mapsto a^t$. 
This proves 
\[
\bigcap_{s\in \bbZ} \theta_s(C)=D_a\cup D_a'.
\]
So Equation \eqref{eq:theo1.2} is equivalent to $\bbP[Y\in D_a\cup D_a']=1$.
In order to prove $\bbP[Y\in D_a]=1$, it remains to prove that 
$\bbP[Y\in D_a'\setminus D_a]=0$.
Note that all function $f\in D'_a\setminus D_a$ is of the form $f(t)=u a^t$, $u>0$ and satisfies $\lim_{t\to -\infty} f(t)=+\infty$. 
Hence,
\[
\bbP[Y\in D_a'\setminus D_a]\leq \bbP[\lim_{t\to-\infty} Y(t)=+\infty].
\]
Equation \eqref{eq:Y} together with Fatou's lemma yields
\[
 \bbE[\liminf_{t\to -\infty}Y(t)]\leq \liminf_{t\to -\infty}\bbE[Y(t)]=1.
\]
We deduce that $\liminf_{t\to -\infty}Y(t)$ is almost surely finite so that 
\[
\bbP[\lim_{t\to-\infty} Y(t)=+\infty]=0. 
\]
Hence $\bbP[Y\in D_a'\setminus D_a]=0$ and $\bbP[Y\in D_a]=1$, which proves Equation \eqref{eq:theo1.1}.

\vspace{0.5cm}
\noindent
{\bf Proof of \eqref{eq:theo1.1} in the case $a=1$:}\\
First we prove that 
\[
\bbP[Y(1)=Y(0)\mid Y(0)>0]=1\quad\mbox{implies}\quad  \bbP[Y(-1)=Y(0)\mid Y(0)>0]=1.
\]
To see this, we note that $\bbP[Y(1)=Y(0)\mid Y(0)>0]=1$ if and only if $\bbP[Y\in C]=1$ with $C=\{f\in\cF; f(0)>0\Rightarrow f(1)=f(0)\}$. 
Since $Y$ is Brown-Resnick stationary, Proposition~\ref{prop:BRcone}  implies $\bbP[Y(\cdot-1)\in C]=1$, which yields $ \bbP[Y(0)=Y(-1)\mid Y(-1)>0]=1$.
Then, Equation \eqref{eq:Y} entails
\[
\bbE\big[Y(-1)1_{\{Y(-1)=Y(0)\}}\big]=\bbE\big[Y(0)1_{\{Y(-1)=Y(0)\}}\big]=1.
\]
We deduce $\bbE\big[Y(0)1_{\{Y(-1)\neq Y(0)\}}\big]=0$ which implies 
\[
\bbP[Y(-1)=Y(0)\mid Y(0)>0]=1. 
\]

Consider the cone 
\[
C=\{f\in\cF;\ f(0)>0\Rightarrow f(1)=f(-1)=f(0)\}.
\]
The conditions
\[
\bbP[Y(1)=Y(0)\mid Y(0)>0]=\bbP[Y(-1)=Y(0)\mid Y(0)>0]=1
\]
implies $\bbP[Y\in C]=1$, whence Proposition \ref{prop:BRcone} yields
\[
\bbP\Big[ Y\in\bigcap_{s\in \bbZ} \theta_s(C) \Big]=1.
\]
Clearly, $\bigcap_{s\in \bbZ} \theta_s(C)$ is equal to the cone of functions
\[
\big\{f\in\cF;\ \forall s\in\bbZ,  f(s)>0 \Rightarrow f(s+1)=f(s-1)=f(s)\big\}.
\]
One can easily prove by induction that this is the cone $D_1$ of constant functions. This proves Equation \eqref{eq:theo1.1}.

\vspace{0.5cm}
\noindent
{\bf Proof of \eqref{eq:theo1.1} in the case $a=0$:}\\
According to Proposition~\ref{prop:indep}, $\bbP[Y(1)=0\mid Y(0)>0]=1$ if and only if $\eta(0)$ and $\eta(1)$ are independent. 
Let $t\geq 2$. By the Markov property, $\eta(0)$ and $\eta(t)$ are independent conditionally on $\eta(1)$. 
But since $\eta(0)$ and $\eta(1)$ are independent, this implies that $\eta(0)$ and $\eta(t)$ are independent. 
Hence $\eta(0)$ and $\eta(t)$ are independent for all $t\geq 1$ and by the stationarity of $\eta$,  
$\eta(t)$ and $\eta(t')$ are  independent for all $t\neq t'$. 
Using Proposition~\ref{prop:indep} again, we deduce
\[
\bbP[Y(t')=0\mid Y(t)>0]=1\quad \mbox{for all}\  t\neq t',
\]
and also
\[
\bbP\big[\forall t'\neq t,\  Y(t')=0  \mid Y(t)>0]=1.
\]
This implies that  the set where $Y$ is non zero has almost surely at most one point. Equivalently,  
$\bbP[Y\in D_0]=1$ and Equation \eqref{eq:theo1.1} is proved.

\section{Continuous time setting}
We consider in this section an extension of Theorem~\ref{theo1} to the continuous time framework.

For $a\in (0,1)$, we denote by $g_a(t)=-\log(a)a^t1_{\{t\geq 0\}}$ the power function. The constant $-\log(a)$ ensures the normalization $\int_\bbR g_a(t)dt=1$.
We consider the moving maximum process
\begin{equation}\label{eq:maxARcont}
 Z_a(t)=\bigvee_{i\geq 1} U_i g_a(t-T_i),\quad t\in\bbR,
\end{equation}
where  $\{(U_i,T_i);\ i\geq 1\}$ is a Poisson point process on $(0,+\infty)\times\bbR$ 
with intensity $u^{-2}dudt$. The time reversed process $\check Z_a$ is defined similarly by
\begin{equation}\label{eq:TRmaxARcont}
 \check Z_a(t)=\bigvee_{i\geq 1} U_i \check g_a(t-T_i),\quad t\in\bbR,
\end{equation}
with $\check g_a(t)=-\log(a)a^{-t} 1_{\{t< 0\}}=g_a(-t^-)$. We use here a slightly different notion
of time reversal so that the function $\check g_a$ is c\`ad-l\`ag.\\
For $a=1$, we define $Z_1=\check Z_1$ a process with constant path and such $Z_1(0)$ has a standard $1$-Fr\'echet distribution.
\begin{lemma}\label{lem:theo2prep}
The processes $Z_a$ and $\check Z_a$ are  stationary simple max-stable processes satisfying the Markov property and with  c\`ad-l\`ag sample paths.
\end{lemma}
\begin{proof}[Proof of Lemma \ref{lem:theo2prep}]
The result is straightforward when $a=1$. For $a\in (0,1)$, the process $Z_a$ is a moving maximum process with shape function $g_a$ 
satisfying $\int_\bbR g_a(t)dt=1$ and is hence a stationary simple max-stable process.\\
Straightforward computations yield that for any $t\in\bbR$
\begin{equation}\label{eq:ARcont}
 Z_a(t+s)=a^s Z_a(t) \bigvee F_a(t,s),\quad s\geq 0,
\end{equation}
with 
\begin{equation}\label{eq:ARcont2}
F_a(t,s)= \bigvee_{i\geq 1} U_i g_a(t+s-T_i)1_{\{T_i>t\}}.
\end{equation}
Note that for $t' \leq t$, $Z_a(t')$ depends only on the points $(U_i,T_i)$ such that $T_i\leq t$ while $F_a(t,s)$ depends only on 
the points $(U_i,T_i)$ such that $t+ s \geq T_i>t$. This implies that  $ (Z_a(t'))_{t'\leq t}$ and $(F_a(t,s))_{s\geq 0}$ are independent 
processes. This together with Equation \eqref{eq:ARcont} implies that the process $Z_a$ satisfies the Markov property. \\
We now prove that $Z_a$ has c\`ad-l\`ag  sample paths. Note that the shape function $g_a$ is c\`ad-l\`ag and satisfies   $\int_{\bbR}\sup_{|z|\leq M}g_{a}(z-t)dt  <\infty$ for all $M>0$.
This implies that the number of points $(U_i,T_i)$ such that $\sup_{|z|\leq M} U_i g_a(z-T_i)>\varepsilon$ is finite since it has a Poisson distribution with mean
\[
\int_0^\infty\int_{\bbR}1_{\{\sup_{|z|\leq M} u g_a(z-t)>\varepsilon\}}u^{-2}dudt=\varepsilon^{-1}\int_{\bbR}\sup_{|z|\leq M}g_{a}(z-t)dt  <\infty.
\]
We deduce that only finitely many functions $U_i g_a(\cdot-T_i)$ contribute to the exceedances of $Z_a(\cdot)=\bigvee_{i\geq 1} U_i g_a(\cdot-T_i)$ above $\varepsilon$ on $[-M,M]$. The function $Z_a\vee\varepsilon$ is thus c\`ad-l\`ag as a maximum of finitely many c\`ad-l\`ag functions. Finally, as $\varepsilon\to 0$, $Z_a\vee\varepsilon$ converges uniformly  to $Z_a$ and $Z_a$ is hence c\`ad-l\`ag as a uniform limit of c\`ad-l\`ag functions.\\
The similar statements for the time reversed process $\check Z_a$ are proved in the same way and we omit the details.
\end{proof}

Theorem \ref{theo1} extends to continuous time processes as follows.
\begin{theorem}\label{theo2}
Any stationary simple max-stable process $\eta=(\eta(t))_{t\in\bbR}$ with c\'ad-l\`ad sample paths and satisfying the Markov property is equal in distribution to $Z_a$ or $\check Z_a$ for some $a\in (0,1]$.
\end{theorem}
Equality in distribution is meant in the sense of equality of laws in the Skohorod space $\cD(\bbR,\bbR)$ with the $J_1$-topology. 
If the max-stable Markov process is $\eta$ is not supposed c\`ad-l\`ag but only continuous in probability, the result still holds in the sense of equality of the finite dimensional distributions.

For the proof of Theorem \ref{theo2}, the following Lemma will be useful.
\begin{lemma}\label{lem:theo2}
For all $\varepsilon>0$, the discrete time process $Z_{a}^{\varepsilon}=(Z_a(\varepsilon t))_{t\in\bbZ}$ is a max-AR(1) process with parameter $a^\varepsilon$.
\end{lemma}
\begin{proof}[Proof of Lemma \ref{lem:theo2}]
The random variable $F_a(t,s)$ given by \eqref{eq:ARcont2} has a $1$-Fréchet distribution with scale parameter $1-a^s$  since
\begin{eqnarray*}
\bbP[F_a(t,s)\leq x]
&=& \bbP\Big[\bigvee_{i\geq 1} U_i g_a(t+ s-T_i)1_{\{T_i>  t\}}\leq x\Big] \\
&=& \exp\Big(-\int_{\bbR}\int_0^{+\infty} 1_{\{u(-\log(a))a^{t+s-\tau}>x\}}1_{\{t<\tau\leq t+s\}} u^{-2}dud\tau\Big)\\
&=& \exp\big(-(1-a^s) /x \big).
\end{eqnarray*}
Using this, one proves easily that the random variables $F_t=F_a(\varepsilon t,\varepsilon)/(1-a^\varepsilon)$, $t\in\bbZ$, are i.i.d. with standard Fréchet distribution. Equation \eqref{eq:ARcont} entails 
\[
 Z_a^{\varepsilon}(t+1)= a^{\varepsilon}Z_a^{\varepsilon}(t) \bigvee (1 - a^{\varepsilon}) F_t ,\quad t\in\bbZ
\]
so that $Z_{a}^{\varepsilon}$ is a max-AR(1) process with parameter $a^\varepsilon$.
\end{proof}

\begin{proof}[Proof of Theorem \ref{theo2}]
The discrete time process $\eta^1=(\eta(t))_{t\in\bbZ}$ extracted from the continuous time process $\eta$ is a stationary simple max-stable process on $\bbZ$ satisfying the Markov property. By Theorem \ref{theo1}, it is equal in distribution either to a max-AR(1) process $X_a$ with $a\in [0,1]$ or a time reversed max-AR(1) process $\check X_a$ with $a\in (0,1)$. 
\begin{itemize}
\item In the case $\eta^1\stackrel{d}= X_a$ with $a\in (0,1]$, we prove that $\eta\stackrel{d}= Z_a$.\\
The process $\eta^{1/n}=(\eta(t/n))_{t\in\bbZ}$ is stationary simple max-stable and Markov. Theorem \ref{theo1} entails that $\eta^{1/n}$ is either a max-AR(1) process $X_{a_n}$ with $a_n\in [0,1]$ or a time reversed max-AR(1) process $\check X_{a_n}$ with $a_n\in (0,1)$. Using the relation $(\eta^1(t))_{t\in\bbZ}=(\eta^{1/n}(nt))_{t\in\bbZ}$, we prove easily that $\eta^{1/n}$ must be a max-AR(1) process with parameter $a_n=a^{1/n}$. Indeed, in all other cases, the process $(\eta^{1/n}(nt))_{t\in\bbZ}$ is not a max-AR(1) process with parameter $a$. By Lemma \ref{lem:theo2}, the process $Z_a^{1/n}$ is also a max-AR(1) process with parameter $a^{1/n}$ so that the processes $(\eta(t/n))_{t\in\bbZ}$ and $(Z_a(t/n))_{t\in\bbZ}$ have the same distribution.
Since this holds true for all $n\geq 1$, we easily see that, for all rational numbers $t_1,\ldots,t_p \in \bbQ$, the random vectors $(\eta(t_1),\ldots,\eta(t_p))$ and $(Z_a(t_1),\ldots,Z_a(t_p))$ have the same distribution. Together with the property that both $\eta$ and $Z_a$ have c\`ad-l\`ag sample paths, this implies that $\eta$ and $Z_a$ have the same distribution in the Skohorod space $\cD(\bbR,\bbR)$(see Billingsley \cite{B68} theorem 14.5).
\item We show that the case $\eta^1\stackrel{d}= X_0$ can not occur.\\
Indeed, if $\eta^1\stackrel{d}= X_0$, it holds also $\eta^{1/n}\stackrel{d}= X_0$ for all $n\geq 1$. This proves that the random variables $\eta(t), t\in\bbQ$ are independent with standard Fr\'echet distribution. This contradicts the fact that $\eta$ has c\'ad-l\`ag sample paths since the difference $\eta(1/n)-\eta(0)$ should  converge in law to zero as $n\to +\infty$.
\item In the case $\eta^1\stackrel{d}= \check X_a$ with $a\in (0,1)$, we prove that $\eta\stackrel{d}= \check Z_a$.\\
Indeed, the time reversed process $\check \eta=(\eta(-t^-))_{t\in\bbZ}$ is then stationary simple max-stable and Markov and such that $(\check \eta(t))_{t\in\bbZ}$ is a max-AR(1) process with parameter $a$. Hence $\check \eta$ and $Z_a$ have the same distribution, whence $\eta$ and $\check Z_a$ have the same distribution.
 \end{itemize}
\end{proof}

\nocite{Leb08}
\nocite{KSdH09}
\nocite{R08}
\nocite{Sm92}

\bibliographystyle{plain}
\bibliography{
Biblio22}
\end{document}